\documentclass[12pt,reqno]{amsart}
\usepackage{amsmath,amssymb,graphicx,epsfig,cite,rotating,setspace,amsthm,color}

\newcommand{\beq}{\begin{equation}}  
\newcommand{\eeq}{\end{equation}}  
\newcommand{\bea}{\begin{eqnarray}}  
\newcommand{\eea}{\end{eqnarray}}

\newcommand\R{{\mathbb{R}}} 
\newcommand\C{{{C^1_0}}}

\newcommand{\red}[1]{{\color{black}{#1}}}

\newcommand\LT{L^2(\mathbb{R})}
\newcommand\LTP{L^2(\mathbb{R}^+)}
\newcommand\LTM{L^2(\mathbb{R}^-)}
\newcommand\HUZ{\widetilde{H}^1}
\newcommand\HT{H^1(\mathbb{R})}

\newcommand{\realpart}[1]{\operatorname{\sf Re\!}\left(#1\right)}
\newcommand{\realpartT}[1]{\operatorname{\sf Re}\left(#1\right)}

\newcommand{\eq}[2]{\begin{equation}\begin{split}#1\end{split}\label{#2}\end{equation}}

\newcommand{\inner}[2]{\left<#1,\,#2\right>}

\newcommand{\eqnn}[1]{\begin{equation}\begin{split}#1\end{split}\nonumber\end{equation}}

\newtheorem{theorem}{Theorem}
\numberwithin{theorem}{section}

\newtheorem{lemma}[theorem]{Lemma}

\newtheorem{proposition}[theorem]{Proposition}

\theoremstyle{definition}
\newtheorem{definition}[theorem]{Definition}
\newtheorem{remark}[theorem]{Remark}

\begin{document}
	
\title[Spectral and linear stability of peakons]{Spectral and linear stability of peakons in the Novikov equation}

\author[S. Lafortune]{St\'ephane Lafortune}
\address[S. Lafortune]{Department of Mathematics, College of Charleston, Charleston, SC 29401, USA}
\email{lafortunes@cofc.edu}


\date{}
\maketitle

\begin{abstract} 
The Novikov  equation is a peakon equation with cubic nonlinearity which, like the Camassa-Holm and the Degasperis-Procesi, is completely integrable.  
In this article, we study the spectral and linear stability of peakon solutions of the Novikov equation.  
	We prove spectral instability of the peakons in $L^2(\mathbb{R})$. To do so, we start with a linearized operator defined on 
	$H^1(\mathbb{R})$ and extend it to a linearized operator defined on weaker functions in  $L^2(\mathbb{R})$. The spectrum of the linearized operator in $L^2(\mathbb{R})$ is proven to cover a closed vertical strip of the complex plane. Furthermore, we prove that the peakons are spectrally unstable on $W^{1,\infty}(\R)$ and linearly and spectrally stable on $\HT$. The result on  $W^{1,\infty}(\R)$ are in agreement with previous work \cite{Chen2019} about linear instability, while our results on $\HT$ are in agreement with the orbital stability obtained in \cite{Chen2021,Palacio2020,Palacios2021}. 
\end{abstract} 

\section{Introduction} 
\label{intro}

We consider the Novikov equation,
\beq \label{Novikov} 
u_t - u_{xxt} +4u^2u_x=3u ^2u_x u_{xx} + u u_{xxx},
\eeq 
which was derived by Novikov \cite{Novikov} as part of a classification of generalized Camassa–Holm-type equations possessing infinite hierarchies of higher symmetries. Equation \eqref{Novikov} has a Lax pair, is bi-Hamiltonian, admits multi-peaked soliton (peakons) solutions, and possesses infinitely many symmetries and conserved quantities \cite{Novikov,HoneNovikovH}. The Novikov equation also admits smooth multi-soliton solutions \cite{Matsuno}. Finally, the Novikov equation was derived in the modeling of the propagation of shallow water waves of ``moderately
large amplitude'' \cite{Chen2022}.

The one-peakon solutions take the form 
\beq \label{peakon}
 u(x,t) = \sqrt{c}\, e^{-|x-ct|} \qquad x \in \mathbb{R}\;\;c>0.
\eeq
This solution is related to the reformulation 
of the Novikov equation \eqref{Novikov} in the weaker form
\begin{equation} 
\label{Novikovi} \vspace{-.2cm}
u_t + u^2 u_x + \frac{1}{2} \phi \ast \left[ \left(\frac{3}{2} uu_x^2 + u^3\right)_x +\frac{1}{2} u_x^3 \right] = 0,
\end{equation}
where $\ast$ denotes convolution and $\phi(x) \equiv e^{-|x|}$ is a Green's 
function satisfying $(1-\partial_x^2) \phi = 2 \delta_0$,
with $\delta_0$ being the Dirac delta distribution centered at $x = 0$.

In this article, we are interested in the stability of the peakon solutions defined on the whole real line. In this situation, the problem of well-posedness has been addressed in \cite{Ni2011,Himonas2012}, while the conditions for global weak and strong solutions on the line has been studied in \cite{Wu2011,Wu2012}. In particular, it is known that solutions with initial conditions $u_0(x)\in H^s(\R), s>3/2$, such that $u_0-u_{0}''>0$, do exist globally.  Finally, the phenomenon of wave-breaking (blow-up) has been studied in \cite{Jiang2012,Wu2012,Chen2016}. 

The peakon solution \eqref{peakon} was shown to be orbitally stable in $\HT$ with the use of a Lyapunov functional in \cite{Chen2021} and asymptotically stable, also on $\HT$, in \cite{Palacio2020,Palacios2021}.  At the linear level, it was shown in \cite{Chen2019} that the $W^{1,\infty}(\R)$ norm of small perturbations increases exponentially in time as $e^t$. 

In \cite{Const5}, it is noted that in the case of the peakon solutions of the Camassa-Holm equation, ``the nonlinearity plays a dominant role rather than being a higher-order correction" and that ``the passage from the linear to the nonlinear theory is not an easy task, and may even be false". We can take from those quotes that the studies of linear and nonlinear stability are two independent problems. This statement is reinforced by the results of \cite{Natali} concerning the  Camassa-Holm peakons. Indeed, in \cite{Natali}, the time exponential growth of the $\HT$ norm of the solutions of the linearized equation is obtained despite the orbital stability of the peakons in $\HT$. In the more general case of the $b$-family, the linear instability of the peakons on $\LT$ has been obtained in \cite{Lafortune2022}. One can surmise that the same phenomenon applies to the Novikov equation and, as mentioned before,  the peakons were shown to be unstable on $W^{1,\infty}(\R)$ \cite{Chen2019}.


The goal of this article is to complete the linear stability study of the peakon solutions of the Novikov equation. We first show that the peakon solutions are spectrally and linearly unstable on $\LT$. We also show that the spectrum of the linear operator consider on  $W^{1,\infty}(\R)$ is in agreement with the linear instability result obtained in \cite{Chen2019}. Finally, we show that the peakons are both linearly and spectrally stable when considered on $\HT$. 

In Section \ref{1} below, we obtain preliminary results, while in sections \ref{2}, \ref{3}, and \ref{4} we obtain, respectively, the results listed in the previous paragraph about the spectral and linear stability in $\LT$, $W^{1,\infty}(\R)$ , and $\HT$. Finally, Section \ref{5} is devoted to conclusions and future projects.

\section{Preliminary results}
\label{1}

In the traveling wave variables $(\xi,t)$, $\xi=x-ct$, we rewrite \eqref{Novikovi} as
\begin{equation} 
\label{Novikovi2} \vspace{-.2cm}
u_t -cu_\xi+u^2 u_\xi + \frac{1}{2} \phi \ast \left[ \left(\frac{3}{2} uu_\xi^2 + u^3\right)_\xi +\frac{1}{2} u_\xi^3 \right] = 0.
\end{equation}
Under the scaling
\eqnn{
u\rightarrow \sqrt{c}u,\;\;t\rightarrow ct,}
the positive parameter $c$ is set to $c=1$.
Eq.~\eqref{Novikovi2} now reads 
\eqnn{
u_t -u_\xi+u^2 u_\xi + \frac{1}{2} \phi \ast \left[ \left(\frac{3}{2} uu_\xi^2 + u^3\right)_\xi +\frac{1}{2} u_\xi^3 \right] = 0,\;\;\xi=x-t,
}
with the one-peakon solution now given by
 \eqnn{
u=\phi\equiv e^{-|\xi|}.
}
In \cite{Chen2019}, it is shown that if $u(x,t)$ is decomposed as the sum of a modulated peakon and its perturbation $v\in \HT$ in the form:
$$
u(x,t)=\phi(\xi-a(t))+v(\xi-a(t),t),
$$
with $a(t)$ satisfying
$$
a'(t)=2v(0,t)+{\mathcal{O}}(v(0,t)^2),
$$
then the equation for $v$ obtained by only keeping the linear terms reads 
\beq 
\label{eigp2}
v_t = (1-\phi^2) v_{\xi} + (\phi v-v_0) \phi'=Lv-v_0\phi',
\eeq
where $v_0\equiv v(0,t)$ and where we defined the linearized operator $L$  is given by 
\beq 
\label{Lop2}
L \equiv (1-\phi^2) \partial_{\xi} + \phi \phi'.
\eeq
In order to extend the linearized equation (\ref{eigp2}) to ${\rm Dom}(L)\subset \LT$, we are going to use an equivalent reformulation of that 
equation (\ref{eigp2}) for $v(\cdot,t) \in H^1(\mathbb{R})$. This is described in Lemma \ref{lem-reformulation2} after the proof of the following elementary properties of $L$.

\begin{lemma}
	\label{lem-projections2}
For $L : {\rm Dom}(L) \subset L^2(\mathbb{R}) \mapsto L^2(\mathbb{R})$, 
\beq
\label{ppp2}
		L\phi=\phi'{\mbox{ and }}L\phi'=\phi.
\eeq	
\end{lemma}
\begin{proof}
Proving \eqref{ppp2} is done by applying $L$ given in \eqref{Lop2}, on $\phi$ and $\phi'$ and using the facts that $\phi'^2=\phi^2$, $\phi''=\phi-2\delta_0$, and 
$(1-\phi^2)\delta_0=0$, with $\delta_0$ being the Dirac delta distribution centered at $\xi = 0$.
\end{proof}

 For the next result, we consider the domain of $L$ on $\HT$ given by
 \beq 
\label{Lop-domain0}
{\mbox{Dom}}(L) = \left\{ v\in H^1(\R): \quad (1-\phi^2) v' \in H^1(\R) \right\}.
\eeq
We have the  following lemma.
\begin{lemma}
	\label{lem-reformulation2}
If $v\in {\mbox{Dom}}(L)\subset \HT$, with ${\mbox{Dom}}(L)$ given in \eqref{Lop-domain0}, is a solution to (\ref{eigp2}), then $v_{0}'(t)=0$, where $v_0\equiv v(0,t)$. Also, $v\in {\mbox{Dom}}(L)\subset \HT$ is a solution 
to the linearized equation (\ref{eigp2}) if and only if $\tilde{v} := 
v - v_0 \phi$ satisfying $\tilde{v}(0,t) = 0$ is a solution of the linearized equation 
\beq 
\label{eigp32}
\tilde{v}_t = L \tilde{v},
\eeq
where $L$ is given in \eqref{Lop2}.
\end{lemma}

\begin{proof}
We prove the first statement of the lemma by taking the limit $\xi \to 0$ 
into (\ref{eigp2}) in the class of functions 
$v \in \HT$, which yields 
\eq{
v_0'(t) = \lim_{\xi \to 0} \left((1-\phi^2) v_{\xi} + (\phi v-v_0) \phi' \right)= 0.
}{vzpz}
Note that we obtain the limit $\lim_{\xi \to 0} \left((1-\phi^2) v_{\xi}\right)=0$ from the {\red{following argument. Since $(1-\phi^2)v_\xi \in H^1(\mathbb{R})$, it is continuous
and it has a limit at zero. The limit must be zero, since otherwise, say it is
$k\neq 0, v_\xi(\xi) \sim k/(2|\xi|)$ as $\xi \sim 0$, which is not square integrable as required by
$v \in H^1(\mathbb{R})$.}}


To prove the second statement, we substitute $v(\xi,t) = \tilde{v}(\xi,t) + v_0(t) \phi(\xi)$ into (\ref{eigp2}) and we obtain
$$
\tilde{v}_t + v_0'(t) \phi = L \tilde{v} + v_0 L \phi-v_0 \phi'.
$$
The last two terms on the RHS cancel out due to the first identity (\ref{ppp2}) in Lemma \ref{lem-projections2}, while the second term on the LHS equals zero because of \eqref{vzpz}. 
The proof in the opposite direction from (\ref{eigp32}) to (\ref{eigp2}) is identical.
\end{proof}

Consider the eigenvalue problem associated to the operator $L$ defined in \eqref{Lop2}
\beq
\label{L0E2}
		(1-\phi^2)v_\xi+ \phi \phi' v = \lambda v, \quad \xi \in \mathbb{R}.
\eeq
The differential equation (\ref{L0E2}) is solved separately for $\xi > 0$ and $\xi < 0$ with the following general solution:
\beq
\label{Lv0S2}
		v(\xi) =
		\left\{ 
		\begin{array}{ll}
 v_+ e^{\lambda \xi} (1-e^{-2\xi})^{(1+\lambda )/2}, & \quad  \xi>0,\\
v_- e^{\lambda \xi} (1-e^{2\xi})^{(1-\lambda)/2}, & \quad \xi<0, \end{array}\right. 
\eeq
	where $v_+$ and $v_-$ are arbitrary constants.

{\red{With the transformations of Lemma \ref{lem-reformulation2}, we have reduced the linearized evolution 
(\ref{eigp2}) defined in $\HT$ to the linearized evolution 
(\ref{eigp32}) defined on the larger space $\LT$. We use this result in two ways in this article. First, in Section \ref{2}, where we see \eqref{eigp32} as a weak version of the original 
linear equation (\ref{eigp2}), just like, for example, \eqref{Novikovi} is a weak version of the Novikov equation \eqref{Novikov} that admits peakon solutions. Since  \eqref{eigp32} is well defined in $\LT$, we use that version in Section \ref{2} to define linear and spectral stability and instability on $\LT$.  Note that this idea was also used in \cite[Lemma 2.3 and Definition 2.4]{Lafortune2022} for the linear and spectral stability analysis of the $b$-family peakons on $\LT$. The second way we use Lemma \ref{lem-reformulation2} is in sections \ref{3} and \ref{4}. In those sections, the function space is $\HT$, and the result of Lemma \ref{lem-reformulation2} applies directly. We use \eqref{eigp32} as a simpler version of the original 
linear equation (\ref{eigp2}), simpler in the sense that it directly relates the spectral properties of $L$ to the solutions of the linear problem.

 As a result of using \eqref{eigp32}, the spectral 
properties of the operator $L$ 
determine the linear stability of the peakons. The spectrum of the operator $L$ 
is defined according to the following standard definition 
(see \cite[Definition 6.1.9 ]{Buhler18}).}}

\begin{definition}
\label{Bigdef}
	Let $A$ be a closed unbounded linear operator on a Banach space $X$ with ${\rm Dom}(A) \subset X$.	The complex plane $\mathbb{C}$ is decomposed into the following two sets:
	\begin{enumerate}
		\item The spectrum of $A$ is given by
		$$
		\sigma(A) =  \left\{\lambda \in \mathbb{C} : \;\;
		A-\lambda I : {\rm dom}(A)\rightarrow X \text{ is not a bijection} \right\}.
		$$
		\item The resolvent
		$$
		\rho(A) =\mathbb{C} \setminus \rho(A).
		$$
		The spectrum is decomposed into the following disjoint sets:
		\begin{enumerate}
			\item the point spectrum
			$$
			\sigma_{\rm p}(A) = \{ \lambda \in \mathbb{C} : \;\; {\rm Ker}(A-\lambda I) \neq \{0\}\},
			$$
			\item the residual spectrum
			$$
			\sigma_{\rm r}(A) = \{ \lambda \in \mathbb{C} : \;\; {\rm Ker}(A-\lambda I) = \{0\}, \;\; \overline{{\rm Ran}(A-\lambda I)} \neq X \},
			$$
			\item the continuous spectrum
			$$
			\sigma_{\rm c}(A) = \{ \lambda \in \mathbb{C} : \;\; {\rm Ker}(A-\lambda I) = \{0\}, \;\; \overline{{\rm Ran}(A-\lambda I)} = X, \;\;
			{{\rm Ran}(A-\lambda I)} \neq X \}.
			$$
		\end{enumerate}\medskip
	\end{enumerate}
	\end{definition}	
	
\section{Spectral and linear instability on $L^2(\R)$}
\label{2}

{\red{The original linearized equation (\ref{eigp2}) is well-defined on $\HT$, but not on $\LT$. However, Equation (\ref{eigp32}), derived from (\ref{eigp2}) as  part of Lemma \ref{lem-reformulation2}, is now well defined 
if $\tilde{v}(\cdot,t) \in {\rm Dom}(L) \subset L^2(\mathbb{R})$. We view  (\ref{eigp32}) as a weak version of Equation (\ref{eigp2})
that we use to define linear stability and instability on $\LT$ as follows.}} 
%

\begin{definition}
	\label{def-instability}
The peakon solution $u(x,t) = \phi(x - t)$ of the Novikov equation \eqref{Novikovi}  is said to be linearly stable on $\LT$ if for every
$\tilde{v}_0 \in {\rm Dom}(L) \subset L^2(\mathbb{R})$, there exists a positive constant $C$ and a unique solution $\tilde{v} \in C(\mathbb{R},{\rm Dom}(L))$ to the linearized equation (\ref{eigp32}) with $\tilde{v}(\xi,0)=\tilde{v}_0(\xi)$ such that
$$
\|\tilde{v}(\cdot,t)\|_{L^2}\leq C \|\tilde{v_0} \|_{L^2}, \quad t>0.
$$
Otherwise, it is said to be linearly unstable.
\end{definition}

We first prove that the peakons are spectrally unstable on $L^2(\R)$ meaning that the spectrum of the operator $L$ has a nontrivial intersection with the open right side of the complex plane. The domain of $L$ on $\LT$ being given by
\beq 
\label{Lop-domain2}
{\mbox{Dom}}(L) = \left\{ v\in L^2(\mathbb{R}) : \quad (1-\phi^2) v' \in L^2(\mathbb{R}) \right\},
\eeq
we prove the following theorem.

\begin{theorem}
	\label{S2}
	The spectrum of the linear operator $L$ defined by (\ref{Lop2}) with domain (\ref{Lop-domain2}) is given by  
	\begin{equation}
	\notag
	\sigma(L) = \left\{ \lambda \in\mathbb{C} : \;\; 0\leq |\realpart{\lambda}|\leq  2 \right\}.
	\end{equation}
Moreover, the point spectrum 
is located for 
\eq{ 0<|\realpart{\lambda}| <  2.}{PS2} 
The continuous spectrum is located for 
$\realpart{\lambda} = 0$ and $\realpart{\lambda} = \pm 2$ and there is no residual spectrum. 
\end{theorem} 

\begin{proof}
{\noindent}{\em Point spectrum of $L$.} We consider  the differential equation \eqref{L0E2} and notice that 
\eq{
{\text{if }} v=v_0(\xi) \text{ solves \eqref{L0E2} for } \lambda=\lambda_0, \text{ then } v=v_0(-\xi) \text{ is a solution for } \lambda=-\lambda_0.}{Sym}
As a consequence, the spectrum of $L$ is symmetric with respect to the real axis. We thus only need to consider values of $\lambda$ such that 
$\realpart{\lambda}\geq 0$.

The solution of the differential equation \eqref{L0E2} is given in \eqref{Lv0S2}. Since 
$\realpart{\lambda}\geq 0$, the conditions for $v$ to go to zero 
as $\xi\rightarrow \pm \infty$ is that $\realpart{\lambda}> 0$ and $v_+=0$. Furthermore, looking at the behavior of $v$ as $\xi\rightarrow 0^-$, the function $v$ is in $\LT$ only if
$
{(1-\realpart{\lambda})}/{2}>-{1}/{2},
$
that is
\eqnn{
0<\realpart{\lambda}<2.
}
Thus the point spectrum is given by \eqref{PS2}.

{\noindent}{\em Resolvent set of $L$.} 
Let us consider the resolvent equation
\beq
\label{L0RE3}
		Lv-\lambda v=f,
\eeq
where $f \in L^2(\mathbb{R})$ is arbitrary. 
The following lemma together with the fact that the 
point spectrum is located in the region defined by \eqref{PS2} imply that the resolvent 
set of $L$ is specified by the condition $|\realpart{\lambda}|>2$.
\begin{lemma}
The resolvent equation \eqref{L0RE3} has a solution in $\LT$ for all $f\in\LT$ if $|\realpart{\lambda}|>2$.  That is, the resolvent set of the operator \eqref{Lop2} on $\LT$ includes the region defined by the relation  $|\realpart{\lambda}|>2$.
\end{lemma}

\begin{proof} 
{\red{For $\realpart{\lambda}>2$, we solve the resolvent equation \eqref{L0RE3} using the expressions in 
\eqref{Lv0S2} and find the following solution converging to zero as $\xi\rightarrow \pm \infty$
\beq
\label{Lv0SH0}
		v(\xi) =
		\left\{ 
		\begin{array}{ll}
v^+_{0}(\xi)\int_\infty^\xi \frac{f(\xi')d\xi'}{(1-e^{-2\xi'})v^+_{0}(\xi')}, & \quad  \xi>0,\\
v^-_{0}(\xi)\int_{0}^\xi \frac{f(\xi')d\xi'}{(1-e^{2\xi'})v^-_{0}(\xi')}, & \quad \xi<0, \end{array}\right.
\eeq
where
$$
v^\pm_{0}\equiv e^{\lambda \xi} (1-e^{\mp 2\xi})^{(1\pm\lambda )/2}.
$$
The two integrals in \eqref{Lv0SH0} exist if $\realpart{\lambda}>2$ because 
\eq{
f\in\LT,\;\;\frac{1}{(1-e^{-2\xi'})v^+_{0}(\xi')}\in L^2(\xi,\infty)\text{ for all }\xi>0,\text{ and }\frac{1}{(1-e^{2\xi'})v^-_{0}(\xi')}\in L^2(\xi,0)\text{ for all }\xi<0.
}{Statements}
The last statement in \eqref{Statements} is seen to be true when $\realpart{\lambda}>2$ by considering the behavior of the expression as $\xi'\to 0^-$ given by  
$$
\frac{1}{(1-e^{2\xi'})v^-_{0}(\xi')}\sim \frac{1}{\xi'^{(3-\lambda)/2}}\text{ as }\xi'\to 0^-.
$$
To show that $v$ as defined in \eqref{Lv0SH0} converges to zero as $\xi\to \pm \infty$, we use l'Hospital's rule. In the case of the limit at $+\infty$, we compute
\eq{
\lim_{\xi\to\infty} v=\lim_{\xi\to\infty}\frac{\left(\int_\infty^\xi \frac{f(\xi')d\xi'}{(1-e^{-2\xi'})v^+_{0}(\xi')}\right)'}{\left(1/v^+_{0}(\xi)\right)'}=
\lim_{\xi\to\infty} \left(\frac{{-f(\xi)}}{\lambda+e^{- 2\xi}}\right)=0,
}{HopEx}
where we used the fact that $|v^+_0|\to\infty$ as $\xi\to\infty$. For the limit at $-\infty$, we obtain
$$
\lim_{\xi\to -\infty} v=\lim_{\xi\to -\infty}\frac{\left(\int_0^\xi \frac{f(\xi')d\xi'}{(1-e^{2\xi'})v^-_{0}(\xi')}\right)'}{\left(1/v^-_{0}(\xi)\right)'}=
\lim_{\xi\to -\infty}\left( \frac{{f(\xi)}}{e^{2\xi}-\lambda}\right)=0,
$$
where we used the fact that $|v^-_0|\to 0$ as $\xi\to\infty$.

Now that we have a solution to the resolvent equation that converges to zero as $\xi\to\pm\infty$ when $\realpart{\lambda}>2$, we show that the solution is in $\LT$.}} We multiply both sides of the resolvent equation \eqref{L0RE3} by $\bar{v}$ and integrate over $\mathbb{R}$. Using the definition of $L$ given in 
	\eqref{Lop2}, one finds
	\eq{
		\inner{\left((1-\phi^2)v\right)'}{v}+3\inner{\phi\phi' v}{v}-\lambda\|v\|^2=\inner{f}{v}.
	}{in222}
By integration by parts, since $\lim\limits_{\xi \to \pm \infty} v(\xi) = 0$ (Equation \eqref{Lv0SH} gives explicit expressions for the solutions to \eqref{L0RE3} that converge to zero), we 
	have that
	$$
	\inner{\left((1-\phi^2) v\right)'}{v}=-\inner{v}{\left((1-\phi^2) v\right)'}-2\inner{\phi\phi' v}{v},
	$$
	thus
\begin{equation}
\label{tech-eq22}
	\realpart{\inner{((1-\phi^2)v)'}{v}}=-\inner{\phi\phi' v}{\bar{v}}.
\end{equation}
	Taking the real part of \eqref{in222}, multiplying by -1, and using \eqref{tech-eq22}, we get 
\beq
\label{inr2}
		-2 \inner{\phi\phi' v}{v}+\realpart{\lambda}\|v\|^2=-\realpart{\inner{f}{v}},
\eeq
where
\beq
\label{inu2}
	-\|v\|^2 \leq \inner{\phi\phi' v}{v}\leq \|v\|^2.
\eeq
Using the upper bound of (\ref{inu2}) in \eqref{inr2}, we find that 
\beq
\notag
		\left(\realpart{\lambda}-2\right)\|v\|^2\leq\left|\realpart{\inner{f}{v}}\right| \leq \| f \| \| v \|,
\eeq
where the Cauchy--Schwarz inequality has been used. Hence, for every $\realpart{\lambda}> 2$, there exists $C_{\lambda}$ such that 
$\| v \| \leq C_{\lambda} \|f \|$ so that this $\lambda$ belongs to $\rho(L)$.

{\red{For the case where $\realpart{\lambda}<0$, the symmetry of the spectrum with respect to the imaginary axis mentioned at the beginning of the proof of Theorem \ref{S2} also applies to the 
resolvent set and thus $\realpart{\lambda} < -2$ belongs to $\rho(L)$.


To finish the proof that $\lambda$ is in the resolvent set if $|\realpart{\lambda}|>2$, we need to show that $v$ is in the domain of $L$ given in \eqref{Lop-domain2}, meaning that we need to show that
$(1-\phi^2)v'\in\LT$. This statement is a direct consequence of the resolvent equation \eqref{L0RE3}, combined with the fact that both $v$ and $f$ are themselves in $\LT$.}}

\end{proof}

{\noindent}{\em Residual spectrum of $L$.}  By Lemma 6.2.6 in \cite{Buhler18}, 
if $\sigma_p(L^*)$ is the empty set, then $\sigma_r(L) \subset \sigma_p(L^*)=\emptyset$, 
where 
\beq
\label{L0a2}
L^* = - \partial_{\xi} (1-\phi^2) + \phi \phi' = -(1-\phi^2) \partial_{\xi} + 3\phi \phi'
\eeq 
is the adjoint operator to $L$ in $L^2(\mathbb{R})$. The eigenvalue problem for $L^*$ 
\beq
\label{EPs2}
-(1-\phi^2) {v}_{\xi} + 3\phi \phi' v = \lambda v.
\eeq
becomes (\ref{L0E2}) after the transformation: $\lambda \mapsto 3\lambda$ and 
$v \mapsto v^{-3}$. Therefore, we obtain the following general solution
by applying this transformation to (\ref{Lv0S2}):
\beq
\label{Lv0A22}
v(\xi) =
\left\{ 
\begin{array}{ll}
v_+ e^{-\lambda \xi} (1-e^{-2\xi})^{-(3+\lambda )/2}, & \quad  \xi>0,\\
v_- e^{-\lambda \xi} (1-e^{2\xi})^{-(3-\lambda)/2}, & \quad \xi<0, \end{array}\right.
\eeq
where $v_+$ and $v_-$ are arbitrary constants. Proceeding similarly with the limits $\xi \to \pm \infty$ and $\xi \to 0^{\pm}$ shows 
that for nonzero solutions of (\ref{EPs2}) to exist in $L^2(\mathbb{R})$ in the case $\realpart{\lambda}>0$, one must 
require $v_- = 0$ and $-(3+\realpart{\lambda})/2>-1/2$. Since these conditions lead to an empty set, this yields $\sigma_r(L)\subset \sigma_p(L^*)=\emptyset$. \\

\end{proof}

Now turning to linear instability, we use the fact that Equation \eqref{eigp32} with initial condition in $\LT$ can be solved exactly by using the method of characteristics. The following proposition on the global existence of the solution on $\LT$ gives the bounds obtained from the exact solution. Notice that the bounds obtained on the $\LT$-norm are in agreement with the result of Theorem \ref{S2} concerning the spectrum of $L$ being located in the band defined by $0\leq |\realpart{\lambda}|\leq 2$. Indeed, we obtain that the $\LT$ norm of the solutions does not grow faster than $e^{2t}$, but there exists a solution for which the norm grows faster than $e^{\lambda_0 t}$ for every $\lambda_0\in (0,2)$.  

\begin{proposition}
	\label{prop-final}
	For every $\mathfrak{v}_0 \in {\rm Dom}(L) \subset \LT$ given in \eqref{Lop-domain2}, the initial-value problem (\ref{eigp32}) with the condition $\tilde v(\xi,0)=\mathfrak{v}_0(\xi)$  admits a unique global solution $v$ for  $t>0$
satisfying the properties that $\| \tilde v(\cdot,t) \|_{L^2(0,\infty)} \leq \| \mathfrak{v}_0 \|_{L^2(0,\infty)}$ and $\| \tilde v(\cdot,t) \|_{L^2(-\infty,0)} \leq e^{2t} \| \mathfrak{v}_0 \|_{L^2(-\infty,0)}$. Furthermore, there exists $\mathfrak{v}_0 \in {\rm Dom}(L)$ such that $\| \tilde v(\cdot,t) \|_{L^2(-\infty,0)} = e^{\lambda_0 t} \| \mathfrak{v}_0 \|_{L^2(-\infty,0)}$ for each $\lambda_0 \in (0,2)$. 
\end{proposition}

\begin{proof}
To prove global existence and the results concerning the $\LT$ norm, we use the solution to \eqref{eigp2} provided by the method of characteristics in \cite{Chen2019} and specialize it to the case \eqref{eigp32}, in which the term $v_0$ is absent.  

The characteristics curves $q(s,t)$ are defined as
\eq{
\left\{
\begin{array}{l}
\displaystyle{
\frac{dq}{dt}=\phi^2(q)-1,}\\ \\
 q(0,s)=s,
 \end{array}\right.
 }{charcur}
 which has a unique solution because $\phi$ is Lipschitz. It follows that
 \eqnn{
 q_s(s,t)=\exp{\left(2\int_0^t \phi(q(s,\tau))\phi_x(q(s,\tau)) d\tau\right)}>0,
 }
 thus $q(.,t)$ is a diffeomorphism on $\R$. Equation \eqref{charcur} is solved exactly in \cite{Chen2019} as
 \eq{
 q(s,t)=
\left\{
\begin{array}{ll}
\displaystyle{\frac{1}{2}\ln\left[1+(e^{2s}-1)e^{-2t}\right],}&\displaystyle{s>0,}\\ \\
\displaystyle{0,}&\displaystyle{s=0,}\\ \\
\displaystyle{-\frac{1}{2}\ln\left[1+(e^{-2s}-1)e^{2t}\right],}&\displaystyle{s<0.}
\end{array}\right.
}{charsol}
Defining
\eq{
V(s,t)\equiv \tilde v(q(s,t),t)),
}{Vdef}
 it is found that
 \eq{
 \left\{
 \begin{array}{l}
 \displaystyle{\frac{dV}{dt}=\phi_x(q)\phi(q)V}\\ \\
\displaystyle{V(s,0)=\mathfrak{v}_0(s).}
  \end{array}
  \right.
  }{dVdt}
  Equation \eqref{dVdt} is solved by the authors of \cite{Chen2019} as
  \eq{
V(s,t)= \left\{
 \begin{array}{ll}
 \displaystyle{\frac{\mathfrak{v}_0(s)}{\sqrt{1+(e^{2t}-1)e^{-2s}}},}&\displaystyle{s>0,}\\ \\
  \displaystyle{\frac{\mathfrak{v}_0(s)}{\sqrt{1+(e^{-2t}-1)e^{2s}}},}&\displaystyle{s<0.}
  \end{array}
  \right.
  }{Vsol}
  The solution given in \eqref{Vsol} is in $\LT$ for all $t\geq 0$ if $\mathfrak{v}_0(s)\in \LT$. Furthermore, since $q(.,t)$ given in \eqref{charsol} is a diffeomorphism of $\mathbb{R}$ with bounded derivative, we have that from \eqref{Vdef}, $\tilde  v(.,t)\in \LT$. The unicity is a consequence of the unicity of the solution of the initial value problem \eqref{dVdt}.
  
  By the chain rule, we obtain 
\eqnn{
\| \tilde v(\cdot,t) \|^2_{L^2(0,\infty)} = \int_0^{\infty} |\mathfrak{v}_0(s)|^2 \left[ 1 + (e^{2t}-1) e^{-2s} \right]^{-2} ds\leq \| \mathfrak{v}_0 \|^2_{L^2(0,\infty)},
}
and 
\eq{
\| \tilde v(\cdot,t) \|^2_{L^2(-\infty,0)} &= \int^0_{-\infty} |\mathfrak{v}_0(s)|^2 \left[ 1 + (e^{-2t}-1) e^{2s} \right]^{-2} ds\\
&\leq \sup_{s\in (-\infty,0)} \left[ 1 + (e^{-2t}-1) e^{2s} \right]^{-2}  \int^0_{-\infty} |\mathfrak{v}_0(s)|^2  ds\\
&=\left( \sup_{s\in (-\infty,0)} \left(\frac{1}{ \left[e^{2t} + (1-e^{2t}) e^{2s} \right]^{2}} \right)\right) e^{4t}\| \mathfrak{v}_0 \|^2_{L^2(-\infty,0)}\\
&= e^{4t}\| \mathfrak{v}_0 \|^2_{L^2(-\infty,0)}.
}{BPel}
To compute the supremum on the third line of \eqref{BPel},
we used the fact that for $s<0$ and $t>0$, we have
\eq{
e^{2t} + (1-e^{2t}) e^{2s}=e^{2s} + (1-e^{2s}) e^{2t}>0,\\
\frac{\partial}{\partial s}\left(e^{2t} + (1-e^{2t}) e^{2s}\right)=2(1-e^{2t}) e^{2s}<0.
}{Pelsc}
As a consequence of \eqref{Pelsc}, the infimum of the expression $e^{2t} + (1-e^{2t}) e^{2s}$ for $s\in (-\infty,0)$, and fixed $t>0$, is realized at $s=0$. This implies that the value of the supremum on the third line of \eqref{BPel} is indeed equal to 1.

If $\mathfrak{v}_0$ is the eigenfunction of the point spectrum of $L$ for the eigenvalue $\lambda_0 \in (0,{2})$, which, from (\ref{Lv0S2}), is given by 
\eqnn{
\mathfrak{v}_0(\xi) ={e^{\lambda_0 \xi}}{(1-e^{2\xi})^{(1-\lambda_0)/2}}\in L^2(-\infty,0),  \quad \xi < 0,
}
then the corresponding solution to the linearization (\ref{eigp32}) takes the form
$$
\tilde v(\xi,t) = \mathfrak{v}_0(\xi) e^{\lambda_0t},\;\;x<0,
$$
so that $\| \tilde v(\cdot,t) \|_{L^2(-\infty,0)} = \| \mathfrak{v}_0 \|_{L^2(-\infty,0)} e^{\lambda_0 t}$. 
\end{proof}

	\section{Spectral and linear instability on $W^{1,\infty}(\R)$}
	\label{3}

	In \cite{Chen2019}, it is shown that the linearization \eqref{eigp2} has a global solution in $C(\mathbb{R},H^1(\R)\cap \C)$ for initial conditions 
$v(\cdot,0)= \mathfrak{v}_0 \in H^1(\R)\cap \C$, where 
\eqnn{
C^1_0\equiv \left\{v\in C(\R)\cap C^1(\R^+)\cap C^1(\R^-): v,v_\xi\in L^\infty(\R)\right\}.
}
The motivation behind the choice of this space stems from a discussion made about the Camassa-Holm peakons in the Introduction section of \cite{Natali}, where it is remarked that no constraint can be made on the value of the derivative of the solution of the linearized equation $v_x$ at the peak. This fact is illustrated by the result of Theorem 3 in \cite{Natali}. In the Novikov case, the authors of \cite{Chen2019} show that while the $H^1$-norm of the solutions of the linearization with initial condition in $H^1(\R)\cap \C$ remain constant in time, the $W^{1,\infty}(\R)$ norm increases as $e^t$. The latter statement establishes the linear instability on $W^{1,\infty}(\R)$ since $C^1_0\subset W^{1,\infty}(\R)$. In this section, 
we complete the work of \cite{Chen2019} on linear instability by obtaining the spectral instability result on $\HT\cap \C$.  

For $v\in H^1(\R)\cap \C$, we can apply the change of variables of Lemma \ref{lem-reformulation2} and use the linearization \eqref{eigp32} with the condition $\tilde{v}(0,t)=0$. The corresponding eigenvalue problem thus is taken in the space 
{\red{\eq{
\HUZ\cap \C\text{ where }\HUZ\equiv \left\{v\in \HT : v(0)=0\right\}.
}{VZdef}}}
In that space, the domain of $L$, which we rewrite here for convenience as
\beq 
\label{Lop23}
L \equiv (1-\phi^2) \partial_{\xi} + \phi \phi',
\eeq
 is given by
\beq 
\label{Lop-domain3}
{\mbox{Dom}}(L) = \left\{ v\in \HUZ\cap \C : \quad (1-\phi^2) v' \in \HUZ\cap \C \right\}.
\eeq
Note that in the definition of the domain in \eqref{Lop-domain3}, we did not have to include the term $\phi\phi' v$. This is because $\phi \phi' v \in \HUZ\cap \C$ if $v\in \HUZ\cap \C$. Indeed,  $\phi\phi' v$ is continuous, it is zero at $\xi=0$, and its derivative can be computed as 
\eq{(\phi\phi' v)'=\phi'^2 v+\phi\phi''v+\phi\phi'v'=2\phi^2v+\phi\phi'v',}{PelCond}
where we used the facts that $v(0)=0$, $\phi'^2=\phi^2$, and $\phi''=\phi+2\delta_0$. The last expression in \eqref{PelCond} is continuous on $\R^+$ and on $\R^-$, it is bounded, and it is in $\LT$, thus proving the claim that $\phi \phi' v \in \HUZ\cap \C$.


We now state and prove the theorem on the spectrum of the operator $L$. The spectrum is characterized by the condition $0<|\realpart{\lambda}|<1$. This is in agreement with
the results of \cite{Chen2019} stating that the $L^\infty$ norm of the first derivative  of the solution to the linearized equation increases as $e^t$.

	\begin{theorem}
	\label{S3}
	The spectrum of the linear operator $L$ defined on $\HUZ\cap \C$ in (\ref{Lop23})--(\ref{Lop-domain3}) is given by  
	\begin{equation}
	\notag
	\sigma(L) = \left\{ \lambda \in\mathbb{C} : \;|\realpart{\lambda} |\leq 1 \right\}.
	\end{equation}
There is no point spectrum, the  inequality $0<|\realpart{\lambda}|<1$ characterizes the band containing the residual spectral, and the continuous spectrum consists of the imaginary axis and the lines $\realpart{\lambda}=\pm 1$.
\end{theorem} 
\begin{proof}
{\noindent}{\noindent}{\em Point spectrum of $L$.} We consider  the differential equation \eqref{L0E2} and assume
$\realpart{\lambda}\geq 0$.
The solution of the differential equation \eqref{L0E2} is given in \eqref{Lv0S2}. Since 
$\realpart{\lambda}\geq 0$, the conditions for $v$ to go to zero 
as $\xi\rightarrow \pm \infty$ is that $v_+=0$ and $\realpart{\lambda}> 0$. Furthermore, looking at the behavior of $v$ as $\xi\rightarrow 0^-$, the function $v$ is in $\HUZ\cap C^1_0$ only if
$
{(1-\realpart{\lambda})}/{2}\geq 1.
$
Since that condition is not compatible with our assumption that $\realpart{\lambda}> 0$, we conclude the $L$ does not have point spectrum on $\HUZ\cap C^1_0$.\\

{\noindent}{\em Residual spectrum of $L$.}  By Lemma 6.2.6 in \cite{Buhler18}, 
if $\sigma_p(L)$ is an empty set, then $\sigma_r(L) = \sigma_p(L^*)$, 
where the adjoint $L^*$ with respect to the $\LT$ pairing is given in \eqref{L0a2}. The solution to the corresponding eigenvalue problem
is given in \eqref{Lv0A22}.
In the case $\realpart{\lambda}\geq0$, the solution goes to zero in the limits $\xi \to \pm \infty$ only if $v_- = 0$ and $\realpart{\lambda}> 0$. 
Furthermore, the solution is in $\tilde{H}^{1*}\cap C^1(\R^+)^*$ only if $-(3+\realpart{\lambda})/2>-2$.
This  leads to the  condition $0<\realpart{\lambda}<1$. Thus $\sigma_r(L)$ is characterized by the condition $0<|\realpart{\lambda}|<1$.  \\

{\noindent}{\noindent}{\em Resolvent set of $L$.} To complete the proof of Theorem \ref{S3}, we prove that the set defined by the inequality $|\realpart{\lambda}|>1$ is included in the resolvent set.
Let us consider the resolvent equation {\red{given in \eqref{L0RE3} with $f \in \HUZ\cap C^1_0$ arbitrary. Also, $\realpart{\lambda} \geq 0$ is assumed without loss of generality.}}

To show that the imaginary axis $\realpart{\lambda}=0$ is not part of the resolvent set, we solve the resolvent equation \eqref{L0RE3} in the case where $\lambda$ is purely imaginary with
$$
f=\frac{(1-e^{-2\xi})^{3/2} e^{\lambda\xi}}{\xi+1},\text{ for }\xi>0.
$$
In that case, the general solution to the resolvent equation \eqref{L0RE3} for $\xi>0$ is given by
\eq{
v=e^{\lambda\xi}(1-e^{-2\xi})^{1/2}\left(\ln(\xi+1)+C\right),
}{SolEss}
which does not converge to zero as $\xi\rightarrow \infty$. Thus, if $\lambda$ is purely imaginary, then ${\rm Ran}(R-\lambda I) \neq X$, implying that the imaginary axis 
is part of the continuous spectrum.

For $\realpart{\lambda}>0$, {\red{we use the expressions given in \eqref{Lv0SH0} that define a solution to  the resolvent equation  \eqref{L0RE3}. For convenience, we rewrite 
the definition of $v$ below as}}
\beq
\label{Lv0SH}
		v(\xi) =
		\left\{ 
		\begin{array}{ll}
v^+_{0}(\xi)\int_\infty^\xi \frac{f(\xi')d\xi'}{(1-e^{-2\xi'})v^+_{0}(\xi')}, & \quad  \xi>0,\\
v^-_{0}(\xi)\int_{0}^\xi \frac{f(\xi')d\xi'}{(1-e^{2\xi'})v^-_{0}(\xi')}, & \quad \xi<0, \end{array}\right.
\eeq
where
$$
v^\pm_{0}\equiv e^{\lambda \xi} (1-e^{\mp 2\xi})^{(1\pm\lambda )/2}.
$$
%
%
%
%

To complete the proof of Theorem \ref{S3}, we prove the following {\red{three lemmas. The first one is a technical lemma which we will use several times in the remainder of this article, while the two others,  used together, show that the region $\realpart{\lambda}>1$ is part of the resolvent set. Lemma \ref{ResLemmaPr} is a preliminary result, while Lemma \ref{ResLemma} directly implies the desired result about the resolvent.

\begin{lemma}
\label{foverxilemma}
If $f\in \HUZ$, with $\HUZ$ defined in \eqref{VZdef}, then $f(\xi)/\xi\in L^2(\R)$. 
\end{lemma}
\begin{proof}
The lemma is a direct consequence of Hardy inequality \cite[Equation (5)]{Hardy2} which states that if $f\in \HT$ with $f(0)=0$, then
$$
\int_0^\infty\frac{f^2(\xi)}{\xi^2}d\xi\leq \int_0^\infty |f'(\xi)|^2 d\xi,\text{ for }f\in C_0^1(0,\infty).
$$
\end{proof}}}

\begin{lemma}
\label{ResLemmaPr}
If $f\in \HUZ$, then the solution $v$ given in \eqref{Lv0SH} is in $\HUZ$  if $\realpart{\lambda}\neq 0$.
\end{lemma}
\begin{proof}
We restrict ourselves to the case $\realpart{\lambda}>0$ since the symmetry \eqref{Sym} can be applied to the resolvent equation \eqref{Lv0SH} if we add the transformation $f(\xi)\rightarrow -f(-\xi)$. Hence, if the lemma is satisfied for $\lambda=\lambda_0$, it is also for $\lambda=-\lambda_0$.

{\em The expression defined in \eqref{Lv0SH} satisfies $\lim_{\xi\to 0}v=0$.} To prove that statement, we compute separately the limits as $\xi\to 0^\pm$. From the right, we rewrite the expression defining $v$ for $\xi>0$ from \eqref{Lv0SH} as
$$ 
v(\xi)=e^{\lambda\xi}(1-e^{-2\xi})^{(1+\lambda)/2}\int_\infty^\xi \frac{f(\xi')d\xi'}{e^{\lambda\xi}(1-e^{-2\xi'})^{(3+\lambda)/2}},\text{ for }\xi>0.
$$
{\red{To compute the limit as $\xi\to 0^+$, we rewrite $v$ as an expression with a denominator that diverges to $\infty$. We then use L'Hospital's rule, as follows}}
\eqnn{
\lim_{\xi\to 0^+} v=\lim_{\xi\to 0^+}\frac{\left(\int_\infty^\xi \frac{f(\xi')d\xi'}{e^{\lambda\xi}(1-e^{-2\xi'})^{(3+\lambda)/2}}\right)'}{\left(e^{-\lambda\xi}(1-e^{-2\xi'})^{-(1+\lambda)/2}\right)'}=
\lim_{\xi\to 0^+} \left(\frac{{-f(\xi)}}{\lambda+e^{- 2\xi}}\right)=0.
}
From the left, the expression for $v$ from \eqref{Lv0SH} can be rewritten as
$$
v(\xi)=e^{\lambda\xi}(1-e^{2\xi})^{(1-\lambda)/2}\int_0^\xi \frac{f(\xi')d\xi'}{e^{\lambda\xi}(1-e^{2\xi'})^{(3-\lambda)/2}},\text{ for }\xi<0.
$$
The integral in the expression above exists for any $\xi<0$, thanks to the fact that $f(\xi)/\xi\in \LT$ by Lemma \ref{foverxilemma}. Thus its limit as $\xi\to 0^-$ is always zero. As a consequence, if $\realpart{\lambda}\leq 1$, it is immediate that $v\to 0$. If $\realpart{\lambda}>1$, the expression in front of the integral is unbounded as $\xi\to 0^-$ and we can apply L'Hospital rule in the following way
$$
\lim_{\xi\to 0^-} v=\lim_{\xi\to 0^-}\frac{\left(\int_0^\xi \frac{f(\xi')d\xi'}{e^{\lambda\xi}(1-e^{2\xi'})^{(3-\lambda)/2}}\right)'}
{\left(e^{-\lambda\xi}(1-e^{2\xi'})^{(\lambda-1)/2}\right)'}=
\lim_{\xi\to 0^-}\left( \frac{{f(\xi)}}{e^{2\xi}-\lambda}\right)=0,\text{ if }\realpart{\lambda}>1.
$$

{\em The expression for $v$ given in \eqref{Lv0SH} is in $\LT$ for all $\realpart{\lambda}>0$.}
To demonstrate that statement, we start by proving that the expression defining $v$ for $\xi>0$ is in $\LTP$ for all $\lambda$ such that $\realpart{\lambda}>0$. To do so, we compute the bound on $|v|$ in the following way
\eqnn{
|v|&\leq \int^\infty_\xi e^{\realpartT{\lambda}(\xi-\xi')}\left(\frac{1-e^{-2\xi}}{1-e^{-2\xi'}}\right)^{(\realpartT{\lambda}+1)/2}\frac{|f(\xi')|}{1-e^{-2\xi'}}d\xi'\\
&\leq \int^0_{-\infty} e^{\realpartT{\lambda} w}\left(\frac{1-e^{-2\xi}}{1-e^{-2(\xi-w)}}\right)^{(\realpartT{\lambda}+1)/2}\frac{|f(\xi-w)|}{1-e^{-2(\xi-w)}}dw,
}
where we made the change of variable $w=\xi-\xi'$. Continuing, we find
\eqnn{
|v|&\leq \sup_{w\in(-\infty,0]} \left(\left(\frac{1-e^{-2\xi}}{1-e^{-2(\xi-w)}}\right)^{(\realpartT{\lambda}+1)/2}\right)\int^0_{-\infty} e^{\realpartT{\lambda} w} \frac{|f(\xi-w)|}{1-e^{-2(\xi-w)}}dw
\\
&= \int^0_{-\infty} e^{\realpartT{\lambda} w} \frac{|f(\xi-w)|}{1-e^{-2(\xi-w)}}dw
\\
&=\left(e^{\realpartT{\lambda} \xi}\chi_ {(-\infty,0]}\right)  \ast \left( \frac{|f(\xi)|}{1-e^{-2\xi}}\right),
}
where $\chi_ {(-\infty,0]}$ is the indicator function.
We use a particular case of Young's convolution inequality
$$
\|f\ast g\|_{\LT}\leq \|f\|_{L^1(\mathbb{R})}\ast \| g\|_{\LT}
$$
to get
\eqnn{
\|v\|_{\LTP}\leq \frac{1}{\realpart{\lambda}} \left\|\frac{f(\xi)}{1-e^{-2\xi}}\right\|_{\LT}.
}
Note that ${f(\xi)}/({1-e^{-2\xi}})$ is in $\LT$ by Lemma \ref{foverxilemma}, since $f$ is in $\HUZ$.


We now prove that the expression defining $v$ for $\xi<0$ in \eqref{Lv0SH} is in $\LTM$. What makes the analysis more tricky is that the exponent $(1-\realpart{\lambda})/2$ used 
in \eqref{Lv0SH} for $\xi<0$ becomes negative when $\realpart{\lambda}<1$. However, the general idea is the same. We first find a bound on $|v|$ as we did in the case $\xi>0$:
\eqnn{
|v|&\leq \int^0_\xi e^{\realpartT{\lambda}(\xi-\xi')}\left(\frac{1-e^{2\xi}}{1-e^{2\xi'}}\right)^{(1-\realpartT{\lambda})/2}\frac{|f(\xi')|}{1-e^{2\xi'}}d\xi'\\
&\leq \int^0_\xi e^{\realpartT{\lambda} w}\left(\frac{1-e^{2\xi}}{1-e^{2(\xi-w)}}\right)^{(1-\realpartT{\lambda})/2}\frac{|f(\xi-w)|}{1-e^{2(\xi-w)}}dw,\;\;w=\xi-\xi',\\
&=\int^0_\xi e^{\realpartT{\lambda} w}\left(\frac{1-e^{2\xi}}{1-e^{2(\xi-w)}}\right)^{(1-\epsilon)/2}\left(\frac{1-e^{2(\xi-w)}}{1-e^{2\xi}}\right)^{(\realpartT{\lambda}-\epsilon)/2}\frac{|f(\xi-w)|}{1-e^{2(\xi-w)}}dw\\
&\leq \sup_{w\in[\xi,0]} \left(\left(\frac{1-e^{2(\xi-w)}}{1-e^{2\xi}}\right)^{(\realpartT{\lambda}-\epsilon)/2}\right)\int^0_\xi \left(\frac{1-e^{2\xi}}{1-e^{2(\xi-w)}}\right)^{(1-\epsilon)/2}e^{\realpartT{\lambda} w} \frac{|f(\xi-w)|}{1-e^{2(\xi-w)}}dw,
}
where $\epsilon$ is any number such that $0<\epsilon\leq \min\left(1,\realpart{\lambda}\right)$, so that the the two exponents in the expression above are positive, $(1-\epsilon)/2>0$ and $(\realpart{\lambda}-\epsilon)/2>0$. We continue
\eq{ 
|v|&\leq \int^0_\xi \frac{e^{\realpartT{\lambda} w}}{(1-e^{2(\xi-w)})^{(1-\epsilon)/2}} \frac{|f(\xi-w)|}{1-e^{2(\xi-w)}}dw\\
&\leq \left(\int^0_\xi \frac{e^{\realpartT{\lambda} w}}{(1-e^{2(\xi-w)})^{1-\epsilon}} dw\right)^{1/2} \left(\int^0_\xi \frac{e^{\realpartT{\lambda} w}|f(\xi-w)|^2}{(1-e^{2(\xi-w)})^2}dw\right)^{1/2},
}{insuccess}
where we used the Cauchy-Schwartz inequality. We rewrite the first integral in the last equation in \eqref{insuccess} as
\eq{
\int^0_\xi \frac{e^{\realpartT{\lambda} w}}{(1-e^{2(\xi-w)})^{1-\epsilon}} dw=e^{\realpartT{\lambda} \xi} \int^0_\xi \frac{e^{-\realpartT{\lambda} (\xi-w)}}{(1-e^{2(\xi-w)})^{1-\epsilon}} dw}{cov}
where we used the change of variable $z=e^{2(\xi-w)}$. From the last expression in \eqref{cov}, it can be deduced 
that the integral is zero when $\xi=0$.
For the limit as $\xi\to-\infty$, we use L'Hospital's rule in the following way
$$
\lim_{\xi\to -\infty}\left(e^{\realpartT{\lambda} \xi} \int^0_\xi \frac{e^{-\realpartT{\lambda} (\xi-w)}}{(1-e^{2(\xi-w)})^{1-\epsilon}} dw\right)
=\lim_{\xi\to -\infty}\frac{\left(\int^0_\xi \frac{e^{-\realpartT{\lambda} (\xi-w)}}{(1-e^{2(\xi-w)})^{1-\epsilon}} dw\right)'}{(e^{-\realpartT{\lambda} \xi})'}
=\lim_{\xi\to -\infty}\frac{1}{\realpartT{\lambda} e^{-\realpart{\lambda} \xi}}=0.
$$

Thus the expression as a whole is continuous and 
bounded for $\xi\in(-\infty,0]$ and we have   
\eqnn{ 
|v|\leq M\left(\int^0_\xi \frac{e^{\realpartT{\lambda} w}|f(\xi-w)|^2}{(1-e^{2(\xi-w)})^2}dw\right)^{1/2},\text{ where }M\equiv \sup_{\xi\in(-\infty,0]} \left(\left(\int^0_\xi \frac{e^{\realpartT{\lambda} w}}{(1-e^{2(\xi-w)})^{1-\epsilon}} dw\right)^{1/2}\right).
}
We then have
\eqnn{
\|v\|_{\LTM}^2&\leq M^2 \left(\int_{-\infty}^0 \int^0_{\xi} \frac{e^{\realpartT{\lambda} w}|f(\xi-w)|^2}{(1-e^{2(\xi-w)})^2}dw\,d\xi\right)\\
&=M^2 \left(\int_{-\infty}^0 e^{\realpartT{\lambda} w}\int^w_{-\infty} \frac{|f(\xi-w)|^2}{(1-e^{2(\xi-w)})^2}d\xi\,dw\right)\\
&\leq M^2\left\| \frac{|f(\xi)|}{|1-e^{2\xi}|}\right\|_{\LT}^2 \left(\int_{-\infty}^0 e^{\realpartT{\lambda}w}dw\right),}
thus
\eqnn{
\|v\|_{\LTM}\leq  \frac{M}{\sqrt{\realpart{\lambda}}}\left\| \frac{|f(\xi)|}{|1-e^{2\xi}|}\right\|_{\LT}. 
}
We note as we did before that ${f(\xi)}/({1-e^{-2\xi}})$ is in $\LT$ by Lemma \ref{foverxilemma} since $f$ is in $\HUZ$.

{\em The expression for $v$ given in \eqref{Lv0SH} is in $H^1(\R)$ for all $\realpart{\lambda}>0$.} To demonstrate that statement, it suffices to prove that
\eq{
\tilde{v}\equiv\frac{v}{1-\phi^2}\in \LT.
}{condLd}
Indeed, since $|\tilde{v}|\leq |v|$, Statement \eqref{condLd} implies that $v \in \LT$. Furthermore, if we solve the resolvent equation \eqref{L0RE3} for $v_\xi$ we find
\eq{
v_\xi=\frac{(\lambda-\phi\phi')v}{1-\phi^2}+\frac{f}{1-\phi^2}.
}{vxi3}
Since the second term on the RHS of \eqref{vxi3} is in $\LT$ by Lemma \ref{foverxilemma}, we have that Statement 
\eqref{condLd} also implies that $v_\xi\in \LT$.


To prove \eqref{condLd}, we start by showing that $v/(1-\phi^2) \in \LTP$ using the part of \eqref{Lv0SH} defining $v$ for $\xi>0$ and we compute
\eqnn{
\frac{v}{1-\phi^2}={e^{\lambda\xi}(1-e^{-\xi})^{(\lambda-1)/2}} \int_\infty^\xi \frac{f(\xi')d\xi'}{e^{\lambda \xi'}(1-e^{-\xi'})^{(3+\lambda)/2}}.
}
Since $v\in\LT$, we are really interested the behavior for small $\xi$, meaning we need only to prove that
\eq{
{\xi^{(\lambda-1)/2}} \int_\infty^\xi \frac{f(\xi')d\xi'}{\xi'^{(3+\lambda)/2}}\in \LTP.
}{enfin4} 
To do so, we prove that an upper bound for the $L^2(\mathbb{R})$ norm of the integral in \eqref{enfin4} is bounded. That is, we prove that  
\eq{
\int_0^\infty \left({\xi^{(\realpartT{\lambda}-1)/2}} \int_\infty^\xi \frac{|f(\xi')|d\xi'}{\xi'^{(3+\realpartT{\lambda})/2}}\right)^2d\xi<\infty.
}{enfin5} 
To show \eqref{enfin5}, we apply the following Hardy inequality {\red{\cite[\S 329]{Hardy}, which states that if $g$ is integrable and non-negative, then
\eq{
\int_0^\infty\left(y^{\alpha-1}\int_y^\infty x^{-\alpha}g(x)dx\right)^pdy\leq \frac{1}{(\alpha+1/p-1)^p}\int_0^\infty g(x)^pdx,\text{ if }\alpha+1/p>1.
}{Hardy1}}}
Using \eqref{Hardy1} with $\alpha=(\realpart{\lambda}+1)/2$ and $p=2$, we find
\eq{
\int_0^\infty \left({\xi^{(\realpartT{\lambda}-1)/2}} \int_\xi^\infty \frac{|f(\xi')/\xi'|d\xi'}{\xi'^{(1+\realpartT{\lambda})/2}}\right)^2d\xi\leq
\frac{4}{\realpart{\lambda}^2}\int_0^\infty\left|\frac{f(\xi)}{\xi}\right|^2d\xi<\infty,
}{enfin6} 
where the last inequality holds because $f(\xi)/\xi\in \LTP$ as a consequence of Lemma \ref{foverxilemma}.

To prove that $v(\xi)/\xi$, with $v$ as given in \eqref{Lv0SH}, is in $\LTM$, we use the expression of $v$ given in the case $\xi<0$ and find that we need to show that
 \eq{
\int_0^\infty \left({\xi^{-(\realpartT{\lambda}+1)/2}} \int_0^\xi \frac{|f(\xi')|d\xi'}{\xi'^{(3-\realpartT{\lambda})/2}}\right)^2d\xi<\infty.
}{enfin7} 
To show \eqref{enfin7}, we apply the following Hardy inequality {\red{\cite[\S 329]{Hardy}}} 
\eqnn{
\int_0^\infty\left(y^{\alpha-1}\int_0^y x^{-\alpha}g(x)dx\right)^pdy\leq \frac{1}{(1-\alpha-1/p)^p}\int_0^\infty g(x)^pdx,\text{ if }\alpha+1/p<1,
}
with $\alpha=(1-\realpart{\lambda})/2$ and $p=2$ to establish that
\eqnn{
\int_0^\infty \left({\xi^{-(\realpart{\lambda}+1)/2}} \int_0^\xi \frac{|f(\xi')/\xi'|d\xi'}{\xi'^{(1-\realpartT{\lambda})/2}}\right)^2d\xi \leq
\frac{4}{\realpart{\lambda}^2}\int_0^\infty\left|\frac{f(\xi)}{\xi}\right|^2d\xi<\infty,
}
which concludes the proof of the lemma.
\end{proof}

\begin{lemma}
\label{ResLemma}
If $|\realpart{\lambda}|> 1$, the solution $v$ given in \eqref{Lv0SH} is in $\HUZ\cap C^1_0$ for all $f \in \HUZ\cap C^1_0$.
\end{lemma}
\begin{proof}
As a consequence of the symmetry \eqref{Sym}, we need only to prove the lemma for the case $\realpart{\lambda}>1$

The expression for $v$ given in \eqref{Lv0SH} is in $\HUZ$ for all $\realpart{\lambda}>0$ by Lemma \ref{ResLemmaPr}.

{\noindent}{\noindent}{\em The derivative of the function $v$ defined in \eqref{Lv0SH} is in $C(\R^+)\cap C(\R^-)$ if $\realpart{\lambda}>1$.} To prove this statement, we solve the resolvent equation \eqref{L0RE3} for $v_\xi$
\eqnn{
v_\xi=\frac{(\lambda-\phi\phi')v+f}{1-\phi^2}.
}
Since $f\in \C$, we have that $f/(1-\phi^2)\in C(\R^+)\cap C(\R^-)$. It thus suffices to prove that $v(\xi)/\xi\in C(\R^+)\cap C(\R^-)$.
From the expression for $v$ given in  \eqref{Lv0SH}, we have that $v(\xi)/\xi$ is continuous everywhere except perhaps at $\xi=0$. We thus only need to prove that $v(\xi)/\xi$ has a limit on the right and on the left as $\xi\to 0$. 
 First, from the right, we obtain
$$
\lim_{\xi\to 0^+} \left(\frac{v(\xi)}{\xi}\right)=\lim_{\xi\to 0^+}\left(\frac{e^{\lambda\xi}(1-e^{-2\xi})^{(1+\lambda)/2}}{\xi}\int_\infty^\xi \frac{f(\xi')d\xi'}{e^{\lambda\xi}(1-e^{-2\xi'})^{(3+\lambda)/2}}\right).
$$
{\red{To compute the limit above, we rewrite the expression as a ratio with a denominator that diverges to $\infty$. We then use L'Hospital's rule, as follows}}
$$
\lim_{\xi\to 0^+} \left(\frac{v(\xi)}{\xi}\right)=
\lim_{\xi\to 0^+}\frac{\left(\int_\infty^\xi \frac{f(\xi')d\xi'}{(e^{\lambda\xi}(1-e^{-2\xi'})^{(3+\lambda)/2}}\right)'}{\left(\xi e^{-\lambda\xi}(1-e^{-2\xi'})^{-(1+\lambda)/2}\right)'}=
\lim_{\xi\to 0^+} \left(\frac{{f(\xi)}}{1-\lambda \xi-e^{-2\xi}(\xi+1)}\right),
$$
{\red{which exists, is finite, and equal to $f(0)$}}, since $f\in\C$. 
From the left, we compute
$$
\lim_{\xi\to 0^-} \left(\frac{v(\xi)}{\xi}\right)=
\left(\frac{e^{\lambda\xi}(1-e^{2\xi})^{(1-\lambda)/2}}{\xi}\int_0^\xi \frac{f(\xi')d\xi'}{e^{\lambda\xi}(1-e^{2\xi'})^{(3-\lambda)/2}}\right).
$$
Since the integral in the expression above converges to zero and the expression in front diverges when $\realpart{\lambda}>1$, we use L'Hospital's rule as 
$$
\lim_{\xi\to 0^-} \left(\frac{v(\xi)}{\xi}\right)=
\lim_{\xi\to 0^-}\frac{\left(\int_0^\xi \frac{f(\xi')d\xi'}{e^{\lambda\xi}(1-e^{2\xi'})^{(3-\lambda)/2}}\right)'}
{\left(\xi e^{-\lambda\xi}(1-e^{2\xi'})^{(\lambda-1)/2}\right)'} v=\lim_{\xi\to 0^-} \frac{{f(\xi)}}{1-\lambda\xi+e^{2\xi}(\xi-1)},
$$
which, again, exists due to the fact that $f\in\C$.

We thus have that $v$ as given in \eqref{Lv0SH} is in $\HUZ$ and its derivative is in $C(\R^+)\cap C(\R^-)$. We thus have proven the lemma.

\end{proof}

\end{proof}


\section{Spectral and linear stability on $H^1(\R)$}
\label{4}

In $H^1(\R)$, we also apply the change of variables of Lemma \ref{lem-reformulation2} and use the linearization \eqref{eigp32} with the condition $\tilde{v}(0,t)=0$. This motivates considering the space {\red{$\HUZ$ defined in \eqref{VZdef}}}  and the linear operator given in (\ref{Lop2}), which we rewrite here for convenience as 
\beq 
\label{Lop3}
L \equiv (1-\phi^2) \partial_{\xi} + \phi \phi'.
\eeq
On $\HUZ$, the domain of $L$ is given by
\beq 
\label{Lop-domain4}
{\mbox{Dom}}(L) = \left\{ v\in \HUZ: \quad (1-\phi^2) v' \in \HUZ \right\}.
\eeq

The following Theorem establishes the spectral stability on $\HT$.  
	\begin{theorem}
	\label{S4}
	The spectrum of the linear operator $L$ defined by (\ref{Lop3})--(\ref{Lop-domain4}) on $\HUZ$ is given by  
	\begin{equation}
	\notag
	\sigma(L) = \left\{ \lambda \in\mathbb{C} : \;\realpart{\lambda} =0 \right\},
	\end{equation}
and it contains continuous spectrum only.
\end{theorem} 

\begin{proof}
{\noindent}{\noindent}{\em Point spectrum of $L$.} The argument that $L$ does not have point spectrum on $\HUZ$ is almost identical to the part of the proof of Theorem \ref{S3}
showing that $L$ does not have point spectrum on $\HUZ\cap C^1(\R^+)$. Like in the proof of Theorem \ref{S3}, we look at the solution of the differential equation \eqref{L0E2} in \eqref{Lv0S2} and conclude that to have a nontrivial solution that goes to zero 
as $\xi\rightarrow \pm \infty$ when $\realpart{\lambda}\geq 0$, we need $v_+=0$ and $\realpart{\lambda}> 0$. Furthermore, looking at the behavior of $v$ as $\xi\rightarrow 0^-$, the function $v$ is in $\HUZ$ only if
$
{(1-\realpart{\lambda})}/{2}>1/2.
$
Since we also have that $\realpart{\lambda}> 0$, we conclude the $L$ does not have point spectrum on $\HUZ$.\\

{\noindent}{\em Residual spectrum of $L$.}  Again, as in the proof of Theorem \ref{S3}, we use the Lemma 6.2.6 in \cite{Buhler18}, 
that says that if $\sigma_p(L)$ is an empty set, then $\sigma_r(L) = \sigma_p(L^*)$, 
where $L^*$  is given in \eqref{L0a2}. We use the expressions given in \eqref{Lv0A22} to conclude that the solution of the eigenvalue problem for $L^*$ is nontrivial and goes to zero in the limits $\xi \to \pm \infty$ only if $v_- = 0$ and $\realpart{\lambda}> 0$. 
Furthermore, the solutions given in \eqref{Lv0A22} is in $H^{1}_0(\R)^*$ only if $-(3+\realpart{\lambda})/2>-3/2$.
Together with the condition $\realpart{\lambda}> 0$, his  leads the conclusion that $\sigma_r(L) = \sigma_p(L^*)=\emptyset$.  \\

{\noindent}{\em The continuous spectrum of $L$ contains the set $i\mathbb{R}$.} By the same argument leading to the expression \eqref{SolEss} in the proof of Theorem \ref{S3}, we obtain that the imaginary axis is indeed contained in $\sigma_c(L)$.\\

{\noindent}{\em The resolvent set of $L$.} The solution $v$ given in  \eqref{Lv0SH} of the resolvent equation  is in $\HUZ$ for all $f\in\HUZ$ if 
$\realpart{\lambda}\neq 0$ by Lemma \ref{ResLemmaPr}. Thus ${{\rm Ran}(A-\lambda I)} = \HUZ$ and $\lambda\notin \sigma_c(L)$ if $\realpart{\lambda}\neq 0$ (see Definition \ref{Bigdef}). Since we established that $\sigma_p(L)=\sigma_r(L)=\emptyset$, we have that the resolvent set consists of the whole complex plane except the imaginary axis.

\end{proof}

We now prove the following result about linear stability on $\HUZ$, in agreement with the spectral stability established by Theorem \ref{S4}.  

\begin{proposition}
	\label{prop-finalH10}
	For every $\mathfrak{v}_0 \in {\rm Dom}(L) \subset \HUZ$ given in \eqref{Lop-domain4}, the initial-value problem (\ref{eigp32}) with the condition $v(\xi,0)=\mathfrak{v}_0(\xi)$ admits a unique global solution $v$ 
satisfying the properties that $\| v(\cdot,t) \|_{\HUZ} = \| \mathfrak{v}_0 \|_{\HUZ}$. 
\end{proposition}
\begin{proof}
We use the solution of the linearized equation found by the characteristic method defined in \eqref{charsol}, \eqref{Vdef}, and \eqref{Vsol}. We already established that the solution is in $\LT$ in the proof of Proposition \ref{prop-final}. To see why the solution is in $\HUZ$ for all $t>0$, we firsat differentiate $V(s,t)$ with respect to $s$ to see that $V$ is in $\HUZ$. Then, we conclude that $v$ given by the relation \eqref{Vdef} itself is in $\HUZ$ due to the fact that $q$ given in \eqref{charsol} is a diffeomorphism of $\R$ with a bounded $s$ derivative. Note that $v(0,t)=0$ since $\mathfrak{v}_0(0)=q(0,t)=0$. The unicity result comes from the fact that the solution \eqref{Vsol} is the unique solution to the equation \eqref{dVdt}.

The result concerning the $\HUZ$ norm is proven by noticing that the argument used to prove that same result in \cite{Chen2019}[Lemma 2.4]  for the case where $\mathfrak{v}_0\in \HT\cap\C$ also applies when $\mathfrak{v}_0\in \HUZ$.  
\end{proof}

\section{Conclusion}
\label{5}

In this article, we studied the spectral and linear stability the peakons of the Novikov equation. We found that they are unstable on $\LT$ and $W^{1,\infty}$, and they are stable on $\HT$. Our results are in agreement with the linear instability results of \cite{Chen2019} on $W^{1,\infty}$ and the orbital stability results of \cite{Chen2021,Palacio2020,Palacios2021} on $\HT$. 

Per the discussion in the Introduction section, there is no easy relation between linear and orbital stability for peakon solutions. For example, consider the case of the $b$-family of peakon equations 
\eqnn{
u_t - u_{xxt} +(b+1)uu_x=bu_x u_{xx} + u u_{xxx}. 
}
This family generalizes the classical cases of the Camassa--Holm equation \cite{Fokas,Cam,Cam2} for $b = 2$ and the Degasperis--Procesi equation \cite{dp,dhh} for $b = 3$. For $b > 1$, numerical simulations in \cite{Holm1,Holm2} showed that 
arbitrary initial data asymptotically resolves into a number of peakons. 
Orbital stability of peakons in the energy space $H^1(\mathbb{R})$ was shown 
for the Camassa-Holm equation in \cite{Const5,Cons1} by using 
conservation of two energy integrals. This method was extended 
in \cite{LinLiu}, where the authors showed orbital stability of 
peakons for the Degasperis-Procesi equation in $L^2(\mathbb{R})$.

On the linear level, the authors of \cite{Natali} proved the time exponential growth of the $\HT$ norm of the solutions of the linearized Camassa-Holm equation. In the more general case of the $b$-family, the linear instability of the peakons on $\LT$ has been obtained in \cite{Lafortune2022}.  The authors of \cite{Char1} considered the spaces $H^s(\R^+)\cap H^s(\R^-)$, $s\geq 1$, thus allowing jump discontinuities at $\xi=0$, and showed that there is unstable point spectrum for $b<1$ for all $s<3/2$. We surmise that the peakons of both the $b$-family and the Novikov equation are linearly and spectrally unstable on such spaces.  

An interesting problem would be to study the stability of the peakons of the Novikov equations under transverse perturbations in a two-dimensional version of the Novikov equation. This is done in the case of smooth soliton solutions of the Camassa-Holm equation in \cite{GLP}. A long term project concerns the peakons of both the $b$-family and their orbital stability for values of $b$ other than 2 or 3. In general, the $b$-family is not integrable and the study of orbital stability is more challenging  for general values of $b$ than in the particular cases of the Camassa-Holm and the Degasperis-Procesi equations.  Furthermore, the Novikov equation was also generalized to a peakon $b$-family \cite{bne}. A possible project concerns the linear and nonlinear stability of the peakons in that case.

{\bf Acknowledgement.} 
The research of S. Lafortune was supported by a Collaboration Grants for Mathematicians from the Simons Foundation (award \# 420847). We thank Dan Maroncelli and D. Pelinovsky for helpful discussions.

\bibliographystyle{unsrt}

\end{document}